\documentclass[11pt]{article}
\usepackage{amsmath,amsthm}

\def\ftoday{le \space\number\day \space\ifcase\month\or
  janvier\or f\'evrier\or mars\or avril\or mai\or juin\or
  juillet\or ao\^ut\or septembre\or octobre\or novembre\or d\'ecembre\fi
  \space\number\year}

\def\real{I\kern-0.20em R}

\def\integer{I\kern-0.20em N}
\def\relative{{\rm \rlap Z\kern 2.2pt Z}}
\def\cc{\kern-.25em{\c c}}

\def\bc{\begin{center}}
\def\ec{\end{center}}
\def\=def{\stackrel{{\rm def}}{=}}

\newcounter{indconst}
\newcounter{auxconst}

\def\bit{\begin{itemize}}
\def\eit{\end{itemize}}
\def\ben{\begin{enumerate}}
\def\een{\end{enumerate}}
\def\bde{\begin{description}}
\def\ede{\end{description}}
 
\def\beq{\begin{equation}}
\def\eeq{\end{equation}}
\def\bfi{\begin{figure}[hbt] \begin{center}}
\def\efi{\end{center} \end{figure}}

\def\bce{\begin{center}}
\def\ece{\end{center}}

\newtheorem {theos} {Theorem}[section]

\newtheorem {coros} {Corollary} [section]

\newtheorem {lemms} {Lemma}[section]
\newtheorem {prop} {Proposition}[section]

\newtheorem {defis} {Definition}[section]

\newtheorem {rems}{Remark} [section]
\newtheorem {ex} {Exemple}[section]
\newtheorem {exs} {Example}[section]

\input amssym.def
\input amssym
\begin{document}

\title{Family of intersecting totally real manifolds of $({\Bbb C}^n,0)$ and germs of holomorphic diffeomorphisms}
\author{Laurent Stolovitch \thanks{CNRS-Laboratoire J.-A. Dieudonn\'e
U.M.R. 7351, Universit\'e de Nice - Sophia Antipolis, Parc Valrose
06108 Nice Cedex 02, France. E-mail : {\tt stolo@unice.fr}. Research of L. Stolovitch was supported by ANR grant ``ANR-10-BLAN 0102'' for the project DynPDE}}
\date{\ftoday}
\maketitle
\def\abstractname{Abstract}
\begin{abstract}
We prove the existence (and give a characterization) of a germ of complex analytic set left invariant by an abelian group of germs of holomorphic diffeomorphisms at a common fixed point.
 We also give condition that ensure that such a group can be linearized holomorphically near the fixed point. It rests on a ``small divisors condition'' of the family of linear parts.
 
The second part of this article is devoted to the study families of totally real intersecting $n$-submanifolds of $(\Bbb C^n,0)$. We give some conditions which allow to straighten
holomorphically the family. If this is not possible to do it formally, we construct a germ of complex analytic set at the
origin which interesection with the family can be holomorphically straightened. The second part is an application of the first. 
\end{abstract}
\section{Introduction}

One of the aim of the article is to study the geometry of some germs of real
analytic submanifolds of $(\Bbb C^n,0)$. We shall consider, in this article, only families of totally real submanifolds of $(\Bbb C^n,0)$ intersecting at the origin. 
We are primarily interested in the holomorphic classification of such objects, that is the orbits of the action of the group of germs of holomorphic diffeomorphisms
fixing the origin. 

In this article, we shall mainly focus on the existence of complex analytic
subsets intersecting such germs of real analytic manifolds. We shall also be interested in the problem of straightening holomorphically the family. We mean that we shall give sufficient condition which will ensure that, in
a good holomorphic coordinates system, each (germ of) submanifold of the family is an $n$-plane. In the case there are formal obstructions to straighten the family, we show the existence of a germ complex analytic variety which intersects the family along a set that can be straightened.
This part of this work takes its roots in and generalizes a work
of Sidney Webster \cite{webster-inter} from which it is very inspired. This part of the work start after having listen to Sidney Webster at the Partial Differential Equations and Several Complex Variables conference held in Wuhan University in June 2004.

The starting point of the first problem appeared already in the work of E. Kasner \cite{kasner1} and was studied,
from the formal view point, by G.A. Pfeiffer \cite{Pfeiffer1}. They were
interested in pairs of real analytic curves in $(\Bbb C,0)$ passing through
the origin. We shall not consider the case were some of the submanifolds
are tangent to some others. We refer the reader to the works of I. Nakai \cite{nakai-toulouse}, J.-M. Tr\'epreau \cite{trepreau-tg} and P. Ahern and X. Gong \cite{gong-ahern-tg} in this direction.

The core of this problem rests on geometric properties of an associated dynamical systems. To be more specific, we shall deal, in the first part of this article, with germs of holomorphic diffeomorphisms of $(\Bbb C^n,0)$ in a neighborhood of the origin (a common fixed point). We shall consider those whose linear part at the origin is different from the identity. The main result is the existence of germ of analytic subset of $(\Bbb C^n,0)$ invariant by an abelian group of such diffeomorphisms under some diophantine condition. This diophantine condition is a Brjuno like condition over small divisors of the full family of linear parts (such condition was already devised by the author for commuting vector fields in \cite{stolo-ihes}). The main result is obtained when trying not to linearize (this not possible in general) but rather to linearize along a well-chosen ideal in a neighborhood of the origin. In fact, this is almost always possible when considering namely the {\it resonant ideal}, generated by the polynomial first integrals of the linear part. The zero locus of this ideal provides, in good holomorphic coordinates, the invariant analytic set. If the family is formally linearizable and the family of linear parts satisfies the small divisor condition then, we shall also prove the the family is holomorphically linearizable (in that case the ideal is chosen to be zero). This first kind (resp. second) of result was obtained by the author for a single (resp. family of commuting) germ of holomorphic vector field at singular point \cite{Stolo-dulac} (resp. \cite{stolo-ihes}). This article corresponds to the two first parts of the preprint \cite{stolo-webster}. This work is also used in a recent work in collaboration with
Xianghong Gong \cite{stolo-gong1}.\\
\section{Abelian group of diffeomorphisms of $(\Bbb C^n,0)$ and their invariant sets}

The aim of this section is to prove the existence of complex analytic invariant subset for a commuting family of germs of holomorphic diffeomorphisms in a neighborhood of a common fixed point. This is very inspired by a previous article of the author concerning holomorphic vector fields. Although the objects are not the same, some of the computations are identical and we shall refer to them when possible.

Let $D_1:=\text{diag}(\mu_{1,1},\ldots,\mu_{1,n}),\ldots, D_l:=\text{diag}(\mu_{l,1},\ldots,\mu_{l,n})$ be diagonal invertible matrices.
Let us consider a family $F:=\{F_i\}_{i=1,\ldots l}$ of {\bf commuting germs of holomorphic diffeomorphisms} of $(\Bbb C^n,0)$ whose linear part, at the origin, is $D:=\{D_ix\}_{i=1,\ldots l}$ : 
$$
F_i(x)=D_ix+f_i(x), \quad \text{with}\quad f_i(0)=0,\;Df_i(0)=0,\;f_i\in {\cal O}_n.
$$

Let ${\cal I}$ be an ideal of ${\cal O}_n$ {\bf generated by monomials} of
$\Bbb C^n$. Let $V({\cal I})$ be the germ at the origin, of the analytic
subset of $(\Bbb C^n,0)$ defined by ${\cal I}$. It is left invariant by
the family $D$. Let us set $\hat {\cal I}:=\widehat{\cal O}_n\otimes {\cal
  I}$. Here we denote ${\cal O}_n$ (resp. $\widehat{\cal O}_n$) the ring of
germ of holomorphic function at the origin (resp. ring of formal power series)
of $\Bbb C^n$. For $Q=(q_1,\ldots, q_n)\in \Bbb N^n$ and
$x=(x_1,\ldots,x_n)\in \Bbb C^n$, we shall write
$$
|Q|:=q_1+\cdots +q_n,\quad x^Q:=x_1^{q_1}\cdots x_n^{q_n}.
$$ 
We shall shall denote $\Bbb N^n_2$, the set of $Q\in\Bbb N^n$ such that $|Q|\geq 2$.

Let $\{\omega_k(D,{\cal I})\}_{k\geq 1}$ be the sequence of positive numbers defined by
$$
\omega_k(D,{\cal I})=\inf\left\{\max_{1\leq i\leq l}|\mu_{i}^Q-\mu_{i,j}|\neq 0\;|\;2\leq |Q|\leq 2^k, 1\leq j\leq n,Q\in \Bbb N^n, x^Q\not\in {\cal I}\right\}
$$
where $\mu_i^Q:=\mu_{i,1}^{q_1}\cdots\mu_{i,n}^{q_n}$.
Let $\{\omega_k(D)\}_{k\geq 1}$ be the sequence of positive numbers defined by
$$
\omega_k(D)=\inf\left\{\max_{1\leq i\leq l}|\mu_{i}^Q-\mu_{i,j}|\neq 0\;|\;2\leq |Q|\leq 2^k, 1\leq j\leq n,Q\in \Bbb N^n \right\}.
$$
\begin{defis}
\begin{enumerate}
\item We say that the ideal ${\cal I}$ is properly embedded if it has a set of monomial generators that does not involve all variables. In that case, the set ${\cal S}$ of variables not involved in any generator is not empty.
\item We say that the family $D$ is {\bf diophantine} (resp. on ${\cal I}$) if
$$
(\omega(D)):\quad-\sum_{k\geq 1}\frac{\ln \omega_k(D)}{2^k}<+\infty \quad (\text{resp.  } (\omega(D,{\cal I})):\quad-\sum_{k\geq 1}\frac{\ln \omega_k(D,{\cal I})}{2^k}<+\infty).
$$
\item A linear anti-holomorphic involution of $\Bbb C^n$ is a map
$\rho(z)=P\bar z$ where the matrix $P$ satisfies $P\bar P=Id$; $\bar z$ denotes the complex conjugate of $z$.
\item We denote by $\widehat{CI}$ the vector subspace of formal power series with no monomial in ${\cal I}$.
\item We shall say that ${\cal I}$ is {\bf compatible} with a anti-holomorphic linear involution $\rho$ if the map $\rho^*:\widehat{\cal O}_{n}\rightarrow \overline{\widehat{\cal O}_{n}}$ defined by $\rho^*(f)=f\circ \rho$  maps  $\widehat{\cal I}$ to $\overline{\widehat{\cal I}}$ and  $\widehat{CI}$ to $\overline{\widehat{CI}}$.
\end{enumerate}
\end{defis}
\begin{rems}
If $n>3$ and ${\cal I}=(x_1x_2)$ then ${\cal I}$ is properly embedded with ${\cal S}:=\{x_3,\ldots, x_n\}$. The ideal ${\cal I}=(0)$ can also be regarded as properly embedded with ${\cal S}:=\{x_1,\ldots, x_n\}$.
\end{rems}
Let $\widehat{\cal O}_n^D$ be the ring of formal invariant of the family $D$, that is 
$$
\widehat{\cal O}_n^D:=\{f\in \widehat{\cal O}_n\,|\;f(D_ix)=f(x)\; i=1,\ldots, l  \}.
$$
Let ${\cal C}_D$ be the non-linear formal centralizer of $D$, that is 
$$
{\cal C}_D:=\{F\in (\widehat{\cal M}^2_n)^n\,|\;F(D_ix)=D_iF(x)\; i=1,\ldots, l  \}.
$$
Here $\widehat{\cal M}_n$ denotes the maximal ideal of the ring $\widehat{\cal O}_n$ of formal power series.
It can be shown (as in proposition 5.3.2 of \cite{stolo-ihes}) that this ring
is generated by a finite number of monomials $x^{R_1},\ldots, x^{R_p}$ and
that the non-linear centralizer ${\cal C}_D$ of $D$ is a module over
$\widehat{\cal O}_n^D$ of finite type. Let $\text{ResIdeal}$ be the ideal
generated by the monomials $x^{R_1},\ldots, x^{R_p}$ in ${\cal O}_n$.
\begin{defis}
We say that the family $F$ is {\bf formally linearizable on
    $\hat{\cal I}$} if there exists a formal diffeomorphism $\hat\Phi$ of
  $(\Bbb C^n,0)$, tangent to the identity at the origin with a zero projection on ${\cal C}_D$ such that
  $\hat\Phi_*F_i-D_ix \in (\hat{\cal I})^n$ for all $1\leq i\leq l$. Here, $\hat\Phi_*F_i$ stands for $\hat \Phi\circ F_i\circ \hat\Phi^{-1}$.
\end{defis}
\begin{theos}\label{theo-invariant}
Let ${\cal I}$ be a monomial ideal (resp. properly embedded). Assume that the family $D$ is diophantine (resp. on
${\cal I}$). If the family $F$ is formally linearizable on $\hat{\cal
  I}$, then it is holomorphically linearizable on ${\cal I}$. Moreover,
there exists a unique such diffeomorphism $\Phi$ such that the projection of the Taylor expansion of $\Phi-Id$ onto ${\cal I}^n\cup {\cal C}_D$ vanishes and which linearizes $F$ on ${\cal I}$.

Moreover, let $\rho$ be a linear anti-holomorphic involution such that $\rho {\cal C}_D\rho={\cal C}_D$. We assume that  ${\cal I}$ is compatible with $\rho$. Assume that, for all $1\leq i\leq l$, $\rho\circ F_i\circ \rho$ belongs to the group generated by the $F_i$'s. Then $\Phi$ and $\rho$ commute with each other.
\end{theos}
This theorem can be rephrased as follow~: Under the afore-mentioned
diophantine condition,  then there exists a germ of holomorphic diffeomorphism
$\Phi$ such that $\Phi_*F_i-D_ix \in ({\cal I})^n$ for all $1\leq i\leq l$. As
a consequence, in a good holomorphic coordinates system, the analytic subset
$V({\cal I})$ is left invariant by each $F_i$ and its restriction to it is the
linear mapping $x\mapsto D_{i|V({\cal I})}x$. Recall that the analytic set is actually the union of complex linear subspaces whose components are obviously invariant by the diagonal linear map.
\begin{rems}
Such a diophantine condition of a family was first devised in the case of germs of vector fields by the author \cite{stolo-ihes}. In order to conjugate a family of commuting diffeomorphisms of the circle, close to rotations, to rotations, J. Moser \cite{moser-circle} also used a Siegel type condition for the full family of rotations. 
\end{rems}
\begin{rems}
The family $D$ can be diophantine while none of the $D_i$'s is (see \cite{Stolo-asi07,stolo-ihes}).
\end{rems}
\begin{rems}
The fact that the ideal is properly embedded allow us to use only the diophantine condition $(\omega(D,{\cal I}))$ which is weaker than $(\omega(D))$
\end{rems}
The second part of the theorem will only be used for further applications.
\begin{rems}\label{invar-facile}
\begin{enumerate}
\item If the ring of invariant of $D$ reduces to the constants and if $D$ is
  diophantine, then $F$ is holomorphically linearizable in a neighborhood of
  the origin. For a single diffeomorphism, this was obtain by H. R\"ussmann
  \cite{russmann-ihes,russmann-diffeo} and by T. Gramtchev and M.Yoshino
  \cite{gram-yoshi} for an abelian group under a slightly coarser diophantine condition.
\item There are examples of germs of diffeomorphisms with non diagonalizable linear part, which cannot be analytically linearized even if the diffeomorphism is formally linearizable and the semi-simple part of its linear part is diophantine\cite{delatte-gramchev,yoccoz-ast}.
\item The existence of an invariant manifold for a germ of diffeomorphism was obtain by J. P\"oschel \cite{Posch}. Despite the fact that we are dealing with a family of diffeomorphisms, the main difference is that we are able to linearize {\bf simultaneously} on each irreducible component of the analytic set.

\end{enumerate}
\end{rems}
According to M. Chaperon \cite{chaperon-ast}[theorem 4, p.132], if the family of diffeomorphisms is abelian then there exists a formal diffeomorphism $\hat \Phi$ such that
$$
\hat \Phi_*F_i(D_jz)= D_j\hat \Phi_*F_i(z),\quad 1\leq i,j\leq l.
$$
We call the family of $\hat \Phi_*F_i$'s a {\bf formal normal form} of the family $F$.
Then we have the following corollary~:
\begin{coros}\label{invar-group}
Let $F$ be an abelian family of germs of holomorphic diffeomorphisms of $(\Bbb C^n,0)$.  Let us assume that $D$ is diophantine on $ResIdeal$. If the non-linear centralizer of $D$ is generated by the $x^{R_i}$'s then $F$ is holomorphically linearizable on ${\cal I}$.
\end{coros}
\begin{coros}\label{lin-group}
Let $F$ be an abelian family of germs of holomorphic diffeomorphisms of $(\Bbb C^n,0)$.  Let us assume that the family $D$ is diophantine. If $F$ is formally linearizable, then $F$ is holomorphically linearizable in $(\Bbb C^n,0)$.
\end{coros}
\begin{rems}
The condition that the non-linear centralizer of $D$ is generated by the $x^{R_i}$'s means: if $\mu_i^Q=\mu_{i,j}$ for some $Q\in \Bbb N^n_2$, $1\leq j\leq n$ and for all $1\leq i\leq l$, then 
$x^Q$ belongs to the ideal generated by $x^{R_1},\ldots, x^{R_p}$. This a very weak condition since only all but a finite number of resonances satisfy this condition.
\end{rems}
\begin{rems}
Since ${\cal I}$ is a monomial ideal, its zero locus is an intersection of unions of complex hyperplanes $\{z_{i_j}=0\}$.
\end{rems}
\begin{exs}
For instance, if the eigenvalues $\mu_1,\mu_2$ of germ of diffeomorphism $\phi$ in $(\Bbb C^2,0)$ satisfy $\mu_1\mu_2=1$ then, in a good holomorphic coordinates system $(z_1, z_2)$ of a neighborhood of the origin, the complex set $\{z_1z_2=0\}$ is invariant under $\phi$ as soon as $\mu_1$ is diophantine. The ideal ${\cal I}=(z_1z_2)$ is not properly embedded.
\end{exs}

We shall prove that there exists a holomorphic map $\Phi : (\Bbb C^n,0)\rightarrow (\Bbb C^n,0)$, tangent to the identity at the origin, such that
$$
\Phi^{-1}\circ F_i \circ \Phi (y)= G_i(y):= D_iy +g_i(y)\quad i=1,\ldots, l
$$ 
where the components of $g_i$ are non-linear holomorphic functions and belong
to the ideal ${\cal I}$. It is unique if we require that its projection on
${\cal I}^n\cup {\mathcal C}_D$ is zero.

Let us set $x_j=\Phi_j(y):= y_j+\phi_j(y)$, $j=1,\ldots, n$. Let us expand the equations $F_i \circ \Phi (y)=\Phi\circ G_i$, $i=1,\ldots,l$.
For all $1\leq j\leq n$ and all $i=1,\ldots,l$, we have
\begin{eqnarray*}
\mu_{i,j}y_j+g_{i,j}(y)+\phi_{j}(G_i(y)) & = & \mu_{i,j}(y_j+\phi_{j}(y))+f_{i,j}(\Phi (y))\\
g_{i,j}(y)+\phi_{j}(D_iy) +(\phi_{j}(G_i(y))-\phi_{j}(D_iy))& = & \mu_{i,j}\phi_{j}(y)+f_{i,j}(\Phi (y))\\
\end{eqnarray*}
We have, by definition, $\phi_{j}(D_iy)=\phi_j(\mu_{i,1}y_1,\ldots, \mu_{i,n}y_n)$.
Let us expand the functions at the origin~:
$$
f_{i,j}(y)=\sum_{Q\in\Bbb N^n_2}f_{i,j,Q}y^Q,\;g_{i,j}(y)=\sum_{Q\in\Bbb N^n_2}g_{i,j,Q}y^Q\text{ and }\phi_{j}(y)=\sum_{Q\in\Bbb N^n_2}\phi_{j,Q}y^Q.
$$
Then we have
\begin{equation}
\sum_{Q\in\Bbb N^n_2}\delta^i_{Q,j}\phi_{j,Q}y^Q + g_{i,j}(y) = f_{i,j}(\Phi (y))-(\phi_{j}(G_i(y))-\phi_{j}(D_iy))\label{conjugaison}
\end{equation}
where 
$$
\delta^i_{Q,j} := \mu_i^Q-\mu_{i,j},\quad\mu_i :=(\mu_{i,1},\ldots,\mu_{i,n}).
$$

Let $\{f\}_Q$ denote the coefficient of $x^Q$ in the Taylor expansion at the origin of $f$.
We define $\phi_{j,Q}$ and $g_{i,j,Q}$ by induction on $|Q|\geq 2$ in the following way :
\begin{itemize}
\item if $y^Q$ does not belong to ${\cal I}$ and $\max_i |\delta^{i}_{Q,j}|\neq 0$, then there exists $1\leq i_0\leq l$ such that $|\delta^{i_0}_{Q,j}|=\max_i |\delta^{i}_{Q,j}|$. We set 
\begin{eqnarray*}
\phi_{j,Q} & = & \frac{1}{\delta^{i_0}_{Q,j}}\left\{f_{i_0,j}(\Phi (y))-(\phi_{j}(G_{i_0}(y))-\phi_{j}(D_{i_0}y))\right\}_Q\\
g_{i,j,Q} & = & 0.
\end{eqnarray*}
\item If $y^Q$ does not belong to ${\cal I}$ and $\max_i |\delta^{i}_{Q,j}|=0$, then we have 
$$
\left\{f_{i_0,j}(\Phi (y))-(\phi_{j}(G_{i_0}(y))-\phi_{j}(D_{i_0}y))\right\}_Q=0
$$
and we set $\phi_{j,Q} = 0 = g_{i,j,Q}$.

\item If $y^Q$ belongs to ${\cal I}$, we set 
\begin{eqnarray*}
\phi_{j,Q} & = & 0\\
g_{i,j,Q} & = & \left\{f_{i,j}(\Phi (y))-(\phi_{j}(G_{i}(y))-\phi_{j}(D_{i}y))\right\}_Q.
\end{eqnarray*}
\end{itemize}
\begin{rems}
Assume that $G_{i}(y)-D_iy$ belongs to $(\widehat {\cal I})^n$. If $y^Q$ does not belong ${\cal I}$ then $\{(\phi_{j}(G_{i}(y))-\phi_{j}(D_{i}y))\}_Q=0$. In fact, let us Taylor expand $\phi_{j}$ at the point $D_{i}y$. We obtain 
$$
\phi_{j}(G_{i}(y))-\phi_{j}(D_{i}y)=\sum_{|\alpha|\geq 1}\frac{1}{\alpha !}D^{\alpha}\phi_{j}(D_{i}y)(G_{i}(y)-D_iy)^{\alpha}\in (\widehat {\cal I})^n.
$$
Moreover, if the $k$-jet of $G_{i}(y)-D_iy$ belongs to $(\widehat {\cal I})^n$ and if $|Q|=k+1$ then $\{(\phi_{j}(G_{i}(y))-\phi_{j}(D_{i}y))\}_Q=0$ since $\phi_j$ is of order $\geq 2$.
\end{rems}
\begin{lemms}
Assume that the family $F$ is formally linearizable on $\hat {\cal I}$. Then the formal diffeomorphism $\Phi$ defined above linearizes simultaneously the family $F$ on $\hat {\cal I}$ where $\hat {\cal I}:=\widehat{\cal O}_n\otimes {\cal I}$.
\end{lemms}
\begin{proof}
Let $\Psi$ be a formal diffeomorphism tangent to identity at the origin and linearizing formally the family $F$ on $\hat {\cal I}$. For all $1\leq i,j\leq j\leq l$, we have 
$$
F_i\circ F_j  =  F_j\circ F_i\;\;\text{ thus }\;\; F_i\circ F_j\circ \Psi  =  F_j\circ F_i\circ \Psi.
$$
Therefore, we have
$$
D_iD_j\Psi+D_i(f_j\circ \Psi) +f_i(D_j\Psi+f_j\circ \Psi) =  D_jD_i\Psi+D_i(f_i\circ \Psi) +f_j(D_i\Psi+f_i\circ \Psi).
$$
For each $i$, let us consider the linear diffeomorphism $\tilde D_i\in \widehat{\cal O}_n^n$ with linear part $D_i$~:
$$ 
\tilde D_i:x\mapsto D_i.x. 
$$
Let us define the operator $T_i(U):= U\circ \tilde D_i-D_iU$ defined on $\widehat{\cal O}_n^n$ to itself.
The previous equality reads
$$
T_i(f_j\circ\Psi)+f_j(D_i\Psi+f_i\circ \Psi)-f_j(\Psi\circ \tilde D_i)= T_j(f_i\circ\Psi)+ f_i(D_j\Psi+f_j\circ \Psi)-f_i(\Psi\circ \tilde D_j)
$$
Moreover, we have $F_i\circ \Psi = \Psi\circ G_i$. Let us write $G_i(z)=D_iz+g_i(z)$, where $g_i$ is a map vanishing at first order at the origin.
Hence, we have
\begin{eqnarray*}
f_i(D_j\Psi+f_j\circ \Psi)-f_i(\Psi\circ \tilde D_j) & = & f_i\circ F_j\circ \Psi -f_i(\Psi\circ \tilde D_j) \\
&= & f_i\circ \Psi\circ G_j-f_i\circ\Psi\circ \tilde D_j \\
& = & D(f_i\circ \Psi)(\tilde D_j)g_j+\cdots\\
\end{eqnarray*}
Assume the $G_i$'s are linear on ${\cal I}$ up to order $k\geq 2$. This means that, for any $1\leq m\leq n$ and any $1\leq i\leq l$, the $k$-jet $J^k(g_{i,m})$ belongs to ${\cal I}$. The previous computation shows that the $(k+1)$-jet of $f_i(D_j\Psi+f_j\circ \Psi)-f_i(\Psi\circ \tilde D_j)$ depends only on the $k$-jet of $g_j$ and belongs to ${\cal I}$. The same is true for $\Psi_{j}(G_{i}(y))-\Psi_{j}(D_{i}y)$. Therefore, if $Q\in \Bbb N^n_2$ with $|Q|=k+1$ is such that $x^Q$ does not belong to ${\cal I}$, then we have 
$$
\{f_j(\Psi\circ \tilde D_i)-D_i(f_j\circ \Psi)\}_Q  =  \{f_i(\Psi\circ \tilde D_j)-D_j(f_i\circ \Psi)\}_Q;
$$
that is, for all $1\leq m\leq n$,
\begin{equation}\label{compat}
(\mu_{i}^Q-\mu_{i,m})\{f_{j,m}\circ \Psi\}_Q  = (\mu_{j}^Q-\mu_{j,m})\{f_{i,m}\circ \Psi\}_Q.
\end{equation}
Let us show by induction on $|Q|\geq 2$ that if $x^Q$ does not belong to ${\cal I}$ and $\max_i |\delta^{i}_{Q,j}|\neq 0$, then $\psi_{j,Q}=\phi_{j,Q}$.
In fact, assume that it is true up to order $k$ and let $|Q|=k+1$, $x^Q\not\in {\cal I}$. According to Taylor expansion, we have
$$
\{f_{i,j}(\Psi (y))\}_Q=\{f_{i,j}(\Phi (y))\}_Q.
$$
Thus, according to $(\ref{compat})$, we have
$$
(\mu_{i}^Q-\mu_{i,m})\{f_{j,m}\circ \Phi\}_Q  = (\mu_{j}^Q-\mu_{j,m})\{f_{i,m}\circ \Phi\}_Q.
$$
If $|\delta^{i_0}_{Q,j}|=\max_i |\delta^{i}_{Q,j}|\neq 0$, then $g_{i,j,Q}:=0$ and
$$
\delta^{i_0}_{Q,j}\phi_{j,Q}:= \{f_{i_0,j}\circ \Phi\}_Q=\{f_{i_0,j}\circ \Psi\}_Q=\delta^{i_0}_{Q,j}\psi_{j,Q}.
$$

This means that equation $(\ref{conjugaison})$ is solved by induction and that $\Phi$ linearizes formally the $F_i$'s on $\hat{\cal I}$.
\end{proof}

Let $\rho$ be a linear anti-holomorphic involution satisfying the assumptions of the
theorem. We have $F_i\circ\Phi = \Phi\circ G_i$ where $G_i$ is linear along $\widehat{\cal I}$. Hence, we have 
$$
(\rho \circ F_i\circ \rho)\circ(\rho\circ\Phi\circ\rho) = (\rho\circ\Phi\circ\rho)\circ (\rho \circ G_i\circ \rho).
$$
Let us set $\tilde F_i:=\rho \circ F_i\circ \rho$. By assumptions, $\tilde F_i$ belongs to the group generated by the $F_i$'s. Since $\rho^*\widehat{\cal I}\subset \overline{\widehat{\cal I}}$, then $\rho \circ G_i\circ \rho$ is a formal
diffeomorphism which is linear on $\widehat{\cal I}$. 
By assumptions, the projection of $\rho\circ\Phi\circ\rho-Id $ onto ${\cal I}\cup {\cal C}_D$ vanishes
identically.  By uniqueness, we have $\rho\circ\Phi\circ\rho=\Phi$ since $\Phi$ linearizes $\tilde F_i$ on ${\cal I}$.

We shall prove, by using the majorant method, that $\Phi$ actually converges
on a polydisc of positive radius centered at the origin. Let us define $\Bbb
N_2^n\setminus \hat{\cal I}$ to be the set of multiindices $Q\in \Bbb N^n$
such that $|Q|\geq 2$ and $x^Q\not\in \widehat{\cal I}$. Let $f=\sum_Q f_Qx^Q$
and $g=\sum_Q g_Qx^Q$ be formal power series. We shall say that $g$ {\it
  dominates} if $|f_Q|\leq |g_Q|$ for all multiindices $Q$. We define $\bar f=\sum_Q |f_Q|x^Q$.

First of all, for all $1\leq j\leq n$ and all $ Q\in \Bbb N_2^n\setminus \hat{\cal I}$ such that $\max_{1\leq i\leq l}|\delta_{Q,j}^i|\neq 0$, we have
$$
|\phi_{j,Q}||\delta_{Q,j}|=|\{f_{i_0(Q),j}(\Phi)\}_Q|\leq \{\bar f_{i_0(Q),j}(y+\bar \phi)\}_Q
$$
where $|\delta_{Q,j}|=\max_{1\leq i\leq l}|\delta_{Q,j}^i|=|\delta^{i_0(Q,j)}_{Q,j}|$. In fact, $ \{f_i\circ \Phi\circ G_j-f_i\circ\Phi\circ D_j \}_Q=0$ whenever $ Q\in \Bbb N_2^n\setminus \hat{\cal I}$. This inequality still holds  if $\max_{1\leq i\leq l}|\delta_{Q,j}^i|= 0$. Let us set 
\begin{itemize}
\item $\delta_{Q}:= \min\{|\delta_{Q,j}|, {1\leq j\leq n}$ such that  $\delta_{Q,j}\neq 0\}$,\\
\item $\delta_{Q}:=0$ if $\max_{1\leq i\leq l}|\delta_{Q,j}^i|= 0$.
\end{itemize}

Let us sum over $1\leq j\leq n$ the previous inequalities. Let us first notice that, since $\phi_{j,Q}=0$ if $\delta_{Q,j}=0$, we have $\delta_Q\phi_{j,Q}\leq \phi_{j,Q}\delta_{j,Q}$ in all cases. Hence, we obtain for all $Q\in \Bbb N_2^n\setminus \hat{\cal I}$,
$$
\delta_Q\sum_{j=1}^n{|\phi_{j,Q}|}\leq\sum_{j=1}^n{|\phi_{j,Q}||\delta_{Q,j}|}\leq \left\{\sum_{j=1}^n\bar f_{i_0(Q,j),j}(y+\bar \phi)\right\}_Q\leq \left\{\sum_{i=1}^l\left(\sum_{j=1}^n\bar f_{i,j}\right)(y+\bar \phi)\right\}_Q.
$$

Since $\sum_{i=1}^l\sum_{j=1}^n f_{i,j}$ vanishes at the origin with its derivative as well, there exist positives constants $a,b$ such that
$$
\sum_{i=1}^l\sum_{j=1}^n{f_{i,j}}\prec \frac{a\left(\sum_{j=1}^n{x_j}\right)^2}{1-b\left(\sum_{j=1}^n{x_j}\right)}.
$$
Since the Taylor expansion of the right hand side has non-negative coefficients, we obtain
$$
\delta_Q\tilde\phi_Q\leq \left\{\frac{a\left(\sum_{j=1}^n{y_j+\tilde \phi}\right)^2}{1-b\left(\sum_{j=1}^n{y_j+ \tilde \phi}\right)}\right\}_Q
$$
where we have set $\tilde\phi_Q:=\sum_{j=1}^n{|\phi_{j,Q}|}$ and $\tilde \phi=\sum_{Q\in \Bbb N_2^n}\tilde\phi_Qx^Q$. Here, we have set $\tilde\phi_Q=0$ whenever $\delta_Q=0$. 

Let us define the formal power series $\sigma(y)=\sum_{Q\in \Bbb N_2^n}{\sigma_Qy^Q}$ as follow :
\begin{eqnarray*}
\forall Q\in \Bbb N_2^n\setminus (\Bbb N_2^n\setminus \hat{\cal I})\;\;\; \sigma_Q & = & 0\\
\forall Q\in \Bbb N_2^n\setminus \hat{\cal I}\;\;\; \sigma_Q & = & \left\{\frac{a\left(\sum_{j=1}^n{y_j+ \sigma}\right)^2}{1-b\left(\sum_{j=1}^n{y_j+ \sigma}\right)}\right\}_Q
\end{eqnarray*}
\begin{lemms}\cite{Stolo-dulac}[Lemme 2.1]
The series $\sigma$ is convergent in a neighborhood of the origin $0\in \Bbb C^n$.
\end{lemms}

Let us define the sequence $\{\eta_Q\}_{Q\in \Bbb N^n_1\setminus \hat{\cal I}}$ of positive number  as follow~:\\
\ben
\item $\forall P\in \Bbb N^n_1\setminus \hat{\cal I}$ such that $|P|=1$, $\eta_{P}  =  1$ ( such multiindices exists. ),
\item $\forall Q\in \Bbb N^n_2\setminus \hat{\cal I}$ with $\delta_Q\neq 0$
$$
\delta_Q \eta_Q  = \max_{{\substack{ Q_j\in \Bbb N^n_1,S\in \Bbb N^n  \\
Q_1+\cdots+Q_p+S=Q  }}}{\eta_{Q_1}\cdots\eta_{Q_p}},
$$
the maximum been taken over the sets of $p+1$, $1\leq p\leq |Q|$, multiindices $Q_1,\ldots,Q_p,S$ such that
$\forall 1\leq j\leq p,\;Q_j\in \Bbb N_1^n,\; |Q_j|<|Q|$, $S\in \Bbb N^n$. These sets are not empty.
\item $\forall Q\in \Bbb N^n_2\setminus \hat{\cal I}$ with $\delta_Q= 0$, $\eta_Q =0$.
\een
This sequence is well defined. In fact, if $Q\in \Bbb N^n_2\setminus \hat{\cal I}$, then there exists
 multiindices $Q_1,\ldots,Q_p,S$ such that $Q=Q_1+\ldots+Q_p+S$, $\forall 1\leq j\leq p,\;Q_j\in \Bbb N_1^n,\; |Q_j|<|Q|,\;S\in \Bbb N^n$.
In this case, $\forall 1\leq j\leq p,\;\;Q_j\in \Bbb N_1^n\setminus \hat{\cal I}$.

The following lemmas are the key points.
\begin{lemms}\cite{Stolo-dulac}[Lemme 2.2]\label{majoration}
For all $Q\in \Bbb N^n_2\setminus \hat{\cal I}$, we have  $\tilde \phi_Q\leq \sigma_Q\eta_Q$.
\end{lemms}
\begin{lemms}\cite{Stolo-dulac}[Lemme 2.3]\label{lemm-divis}
There exists a constant $c>0$ such that $\forall Q\in \Bbb N_2^n\setminus \hat{\cal I},\;\;\eta_Q\leq c^{|Q|}$.
\end{lemms}

Let $\theta>0$ be such that 
\begin{equation}\label{theta}
4\theta:=\min_{1\leq j\leq n}\max_{1\leq i\leq l}|\mu_{i,j}|\leq 1.
\end{equation}
We can always assume this, even if this means using, for each $1\leq i\leq l$, the inverse of the  diffeomorphisms $F_i$~: Fix $j$, for each $i$, if $|\mu_{i,j}|>1$ then consider $\tilde F_i:=F_i^{-1}$. Then, at the end, we have $|\mu_{i,j}|\leq 1$ for all $i$.
If the ideal ${\cal I}$ is properly embedded, then we shall set 
$$
4\theta:=\min_{j\in{\cal S}}\max_{1\leq i\leq l}|\mu_{i,j}|\leq 1
$$ 
where ${\cal S}\neq \emptyset$ denotes the set of variables not involved in any generator. In particular, we have the property that if $x^Q\not\in {\cal I}$ then $x_sx^Q\not\in {\cal I}$ for all $s\in {\cal S}$. As the in previous case, this can always be achieved.

By definition,  $\eta_Q$ is a product of $1/\delta_{Q'}$ with $|Q'|\leq|Q|$. Let $k$ be a non-negative integer. Let us define $\phi^{(k)}(Q)$ (resp. $\phi^{(k)}_j(Q)$) to be the number of $1/\delta_{Q'}$'s present in this product and such that $0\neq \delta_{Q'}<\theta\omega_k(D,{\cal I})$ (resp. and $\delta_Q=\delta_{j, Q}$). The lemma is a consequence of the following proposition
\begin{prop}\cite{Stolo-dulac}[lemme 2.8]\label{estim-bruno}
For all $Q\in {\Bbb N}^n_2\setminus \hat{\cal I}$, we have $\phi^{(k)}(Q)\leq 2n\frac{|Q|}{2^k}$ if $|Q|\geq 2^k+1$; and $\phi^{(k)}(Q)=0$ if $|Q|\leq 2^k$.
\end{prop}

In fact, $\phi^{(k)}(Q)$ bounds the number of $1/\delta_{Q'}$'s appearing in the product defining $\eta_Q$ and such that $\theta\omega_{k+1}(D, {\cal I})\leq \delta_{Q'}<\theta\omega_k(D, {\cal I})$.
\begin{proof}[Proof of lemma \ref{lemm-divis}]
Let $r$ be the integer such that $2^r+1\leq |Q|<2^{r+1}+1$. Then we have
$$
\eta_Q\leq \prod_{k=0}^r{\left(\frac{1}{\theta\omega_{k+1}(D, {\cal I})}\right)^{\phi^{(k)}(Q)}}.
$$
By applying the Logarithm and proposition \ref{estim-bruno}, we obtain
\begin{eqnarray*}
\ln \eta_Q &\leq &\sum_{k=0}^l{2n\frac{|Q|}{2^k}\left(\ln \frac{1}{\theta\omega_{k+1}(D,{\cal I})}\right)}\\
& \leq & |Q|\left(-2n\sum_{k\geq 0}{\frac{\ln \omega_{k+1}(D,{\cal I})}{2^k}}+2n\ln \theta^{-1}\sum_{k\geq 0}{\frac{1}{2^k}}\right).
\end{eqnarray*}
Since the family $D$ is diophantine, we obtain 
$\eta_Q \leq  c^{|Q|}$ for some positive constant $c$.
\end{proof}

For any positive integer $k$, for any $1\leq j\leq n$, let us consider the function defined on ${\Bbb N}^n_2\setminus \hat{\cal I}$ to be
$$
\forall Q\in \Bbb N_2^n\setminus \hat{\cal I},\;\;\;
\psi^{(k)}_j(Q)=\left \{ \begin{array}{l} 1 \;\;\;\;\mbox{ if }\; \delta_Q=|\delta_{Q,j}|\neq 0 \;\mbox{ and }\; |\delta_{Q,j}|<\theta\omega_k(D, {\cal I})\\
0 \;\;\;\;\mbox{ if }\; \delta_Q=0 \;\mbox{ or }\;\delta_Q\neq |\delta_{Q,j}|\;\mbox{ or }\; |\delta_{Q,j}|\geq \theta\omega_k(D, {\cal I})\\ \end{array} \right .
$$
Then we have, 
$$
0\leq \phi_j^{(k)}(Q)  = \psi^{(k)}_j(Q)+\max_{{\substack{ Q_j\in \Bbb N^n_1,S\in \Bbb N^n  \\
Q_1+\cdots+Q_p+S=Q  }}}{\left(\phi_j^{(k)}(Q_1)+\cdots+\phi_j^{(k)}(Q_p)\right)}.
$$
The proof of proposition \ref{estim-bruno} identical to the proof of \cite{Stolo-dulac}[lemme 2.8] except that we have to use the following version of \cite{Stolo-dulac}[lemme 2.7]~:
\begin{lemms}\label{lemm-tech}
Let $Q\in {\Bbb N}^n_2\setminus \hat{\cal I}$ be such that $\psi^{(k)}_j(Q)=1$.
If $Q=P+P'$ with $(P,P')\in \Bbb N^n_1\times\Bbb N^n_2$ and $|P|\leq 2^k-1$, 
then $(P,P')\in \Bbb N^n_1\setminus \hat{\cal I}\times\Bbb N^n_2\setminus \hat{\cal I}$ and $\psi^{(k)}_j(P')=0$.
\end{lemms}

\begin{proof}  Clearly, if $Q=P+P'\in {\Bbb N}^n_2\setminus \hat{\cal I}$ then $(P,P')\in \Bbb N^n_1\setminus \hat{\cal I}\times\Bbb N^n_2\setminus \hat{\cal I}$.
There are two cases to consider :
\ben
\item if $\delta_{P'}\neq |\delta_{j,P'}|$ or $\delta_{P'}=0$ then $\psi^{(k)}_j(P')=0$, by definition.
\item if $\delta_{P'}=|\delta_{j,P'}|\neq 0$, assume
that $\delta_{P'}<\theta\omega_k(D, {\cal I})$. Then, for all $1\leq i\leq l$, we have
$$
|\mu_i^{P'}|>|\mu_{i,j}|-\theta\omega_k(D, {\cal I})\geq 4\theta -2\theta= 2\theta.
$$
It follows that, for all $1\leq i\leq l$,
\begin{eqnarray*}
2\theta\omega_k(D, {\cal I}) & > & |\mu_i^Q-\mu_{i,j}|+|\mu_i^{P'}-\mu_{i,j}|\\
& > & |\mu_i^Q-\mu_i^{P'}| = |\mu_i^{P'}||\mu_i^{P}-1|.
\end{eqnarray*}
If ${\cal I}$ is properly embedded, for all $a\in {\cal S}$, we have $x_ax^P\not\in {\cal I}$. 
Therefore,  for all $1\leq i\leq l$ and all $a\in {\cal S}$, we have
$$
2\theta\omega_k(D, {\cal I})  >  2\theta|\mu_{i,a}|^{-1}|\mu_i^{P+E_a}-\mu_{i,a}|\\
$$
So, for a well chosen $i=i(a)$, we have $|\mu_i^{P+E_a}-\mu_{i,a}|=\max_j|\mu_j^{P+E_a}-\mu_{j,a}|\geq \omega_k(D, {\cal I})$. So, for that $i(a)$, we have
$$
2\theta\omega_k(D, {\cal I})>  2\theta|\mu_{i(a),a}|^{-1}\omega_k(D, {\cal I}).
$$
This contradicts the facts that $\min_{a\in {\cal S}}\max_{1\leq i\leq l}|\mu_{i,a}|\leq 1$. 
If ${\cal I}$ is not properly embedded, then for each $1\leq a\leq n$ there exists $1\leq i(a)\leq l$ 
$$
2\theta\omega_k(D) > 2\theta|\mu_{i(a),a}|^{-1}\omega_k(D).
$$
Hence, we have $|\mu_{i(a),a}|>1$ which contradicts $(\ref{theta})$.
\een
Hence, we have shown that $\psi^{(k)}_j(P')=0$. 
\end{proof}
\section{Family of totally real $n$-manifolds in $(\Bbb C^n,0)$} 

Let us consider a family $M:=\{M_i\}_{i=1,\ldots, m}$ of real analytic totally real $n$-submanifolds of $\Bbb C^n$ passing through the origin. Locally, each $M_i$ is the fixed point set of an anti-holomorphic involution $\rho_i$~: $M_i = FP(\rho_i)$ and $\rho_i\circ \rho_i=Id$. This means that
$$
\rho_i(z) := B_i\bar z + R_i(\bar z)
$$
where $R_i$ is a germ of holomorphic map at the origin with $R_i(0)=0$ and $DR_i(0)=0$. Each matrix $B_i$ is invertible and satisfies $B_i\bar B_i=Id$.
The tangent space, at the origin, of $M_i$ is the totally real $n$-plane
$$
\{z=B_i\bar z\}
$$
We assume that these are all distinct one from another.
Their intersection at the origin is the set
$$
\left\{z\in \Bbb C^n\;|\;B_i\bar z=z,\;i=1,\ldots,m\right\}\subset \left\{z\in \Bbb C^n\;|\;B_i\bar B_j z=z,\;i,j=1,\ldots,m\right\}.
$$
It is contained in the common eigenspace of the $B_i\bar B_j$'s associated to the eigenvalue $1$. {\bf We shall not assume that this space is reduced to $0$}.

Let us consider the group $G$ generated by the germs of holomorphic diffeomorphisms of $(\Bbb C^n,0)$
$F_{i,j}:=\rho_i\circ \rho_j$, $1\leq i,j\leq m$. Let $D_{i,j}:=B_i\bar B_j$ be the linear part at the origin of $F_{i,j}$. Let us set 
$$
F_{i,j}:=D_{i,j}z+f_{i,j}(z)
$$
where $f_{i,j}$ is a germ of holomorphic function at the origin with $f_{i,j}(0)=0$ and $Df_{i,j}(0)=0$.
\begin{lemms}
We have, for any $1\leq i,j\leq m$,
\begin{equation}
R_i(\bar z) - D_{i,j}R_i(\bar F_{i,j}) =  D_{i,j}B_i\bar f_{i,j}(\bar z)+f_{i,j}(\rho_j)\label{conj}.
\end{equation}
\end{lemms}
\begin{proof}
Let us write the relation $F_{i,j}=\rho_i\circ \rho_j$ and $\rho_i\circ \rho_i=Id$. We obtain
\begin{eqnarray}
f_{i,j}(z) & = & B_i\bar R_j(z) + R_i(\bar \rho_j)\label{equ1}\\
0 & = & B_i\bar R_i(z) + R_i(\bar \rho_i).\label{equ2}
\end{eqnarray}
By multiplying the first equation by $\bar B_i$, we obtain
$$
\bar R_j(z) = \bar B_i f_{i,j}(z)- \bar B_i R_i(\bar \rho_j).
$$
Hence,we have
$$
0  =   B_j\bar B_i f_{i,j}(z)- B_j\bar B_i R_i(\bar \rho_j)+  B_i \bar f_{i,j}(\bar \rho_j)-  B_i \bar R_i(\rho_j\circ\rho_j).
$$
Let us multiply by $\bar B_i$ on the left and take the conjugation. We obtain
$$
0= D_{i,j}B_i\bar f_{i,j}(\bar z)-D_{i,j}B_i \bar R_i(\rho_j)+ f_{i,j}(\rho_j)-R_i(\bar z).
$$
On the other hand, by evaluating equation $(\ref{equ2})$ at $\bar \rho_j$, we obtain
$$
0=B_i\bar R_i(\bar \rho_j) + R_i(\bar F_{i,j}).
$$
At the end,we obtain
$$
R_i(\bar z) - D_{i,j}R_i(\bar F_{i,j}) =  D_{i,j}B_i\bar f_{i,j}(\bar z)+f_{i,j}(\rho_j).
$$
\end{proof}
\begin{defis}
The $\rho_i$'s are simultaneously normalizable whenever $R_i(\bar z) - D_{i,j}R_i(\bar D_{i,j}\bar z)=0$ for all $1\leq i,j\leq l$.
\end{defis}
\begin{rems}\label{rem-normalisation}
If the group $G$ is holomorphically linearizable at the origin then the $\rho_i$'s are simultaneously normalizable. This follows from $(\ref{conj})$ with $f_{i,j}\equiv 0$ for all $i,j$.

Moreover, assume the $D_{i,j}$'s are simultaneously diagonalizable and let us set $D_{i,j}=\text{diag}(\mu_{i,j,k})$. Then, for any $1\leq k\leq n$ and any $1\leq j\leq m$, the $k$-component $\rho_{i,k}$ of $\rho_i$ can be written as
$$
\left(\rho_{i}(z)-B_i\bar z\right)_k = \sum_{\substack{Q\in \Bbb N^n_2\\ \forall j,\;\bar\mu_{i,j}^Q=\mu_{i,j,k}^{-1}}} \rho_{i,k,Q}\bar z^Q.
$$
Here, $(f)_k$ denotes the k-th-component of $f$.
\end{rems}
As a consequence, we have
\begin{theos}
Let us assume that the group $G$ associated to the family of totally real
submanifolds $M$ is a semi-simple Lie group. Then the $\rho_i$'s are
simultaneously and holomorphically normalizable in a neighborhood of the origin.
\end{theos}
\begin{proof}
It is classical \cite{Kushnirenko, stern-semi-simple, Ghys} that if the Lie
group $G$ of germs of diffeomorphisms at a common fixed point is semi-simple then it is holomorphically linearizable in a neighborhood of the origin. Then, apply the previous remark \ref{rem-normalisation}.
\end{proof}
\begin{defis}
We shall say that such a family $M=\{M_i\}_{i=1,\ldots, m}$ of totally real $n$-submanifold of $(\Bbb C^n,0)$ intersecting at the origin is {\bf commutative} if the group $G$ is abelian.
\end{defis}
From now on, we shall assume that {\bf $M$ is commutative} and that the {\bf family $D$}
of linear part of the group $G$ at the origin {\bf is diagonal}. In other words, $D_{i,j}=\text{diag}(\mu_{i,j,k})$.
Let ${\cal I}$ be a monomial ideal of ${\cal O}_n$. It is generated by some
monomials $x^{R_1},\ldots, x^{R_p}$. We shall denote $\bar{\cal I}$ the ideal
of $\Bbb C[[\bar x_1,\ldots,\bar x_n]]$ generated by $\bar x^{R_1},\ldots,
\bar x^{R_p}$.
\begin{defis}
\begin{enumerate}
\item We shall say that the family $M$ of manifolds is {\bf non-resonant} whenever, for all $1\leq i\leq m$,  $1\leq k\leq n$ and for all $Q\in \Bbb N_2^n$, there exists a $1\leq j\leq m$ such that $\bar\mu_{i,j}^Q\neq \mu_{i,j,k}^{-1}$.
\item We shall say that the family $M$ of manifolds {\bf non-resonant on ${\cal I}$} whenever for all monomial $z^Q$ not belonging to ${\cal I}$ and for all couple $(i,k)$, there exists $j$ such that $\bar\mu_{i,j}^Q\neq \mu_{i,j,k}^{-1}$.
\end{enumerate}
\end{defis}
\begin{theos}\label{theo-main}
Assume that the group $G$ is abelian. Let ${\cal I}$ be
a monomial ideal (resp. properly embedded) left invariant by the family $D:=\{D_{i,j}\}$ and the involutions $z\mapsto B_i\bar z$. Assume that $D$ is diophantine (resp. on ${\cal I}$) and that $M$ is non-resonant on ${\cal I}$. Assume $G$ is formally linearizable on ${\cal I}$.
Then, the family $F$ is holomorphically linearizable on ${\cal I}$. Moreover, in these coordinates, the $\rho_i$'s are linear and anti-holomorphic on $\bar{\cal I}$.
\end{theos}
\begin{proof}
By theorem \ref{theo-invariant}, the family $F$ is holomorphically linearized on ${\cal I}$. Let us show that, in these coordinates, the $\rho_i$'s are anti-linearized on $\bar{\cal I}$.

Let us prove by induction on $|Q|\geq 2$ that $\{\rho_{i,k}\}_Q=0$ whenever
$z^Q$ doesn't belong to ${\cal I}$ and
$\bar\mu_{i,j}^Q\neq\mu_{i,j,k}^{-1}$. We recall that $\{\rho_{i,k}\}_Q$
denotes the coefficient of $\bar z^Q$ in the Taylor expansion of
$\rho_{i,k}$. Assume it is case up to order $k$. Let $Q\in \Bbb N_2^n$ with $|Q|=k+1$. Let us compute $\{\rho_{i,k}\}_Q$. Using equation $(\ref{conj})$, we obtain
\begin{eqnarray*}
R_{i}(\bar z) - D_{i,j}R_{i}(\bar D_{i,j}\bar z)& = & D_{i,j}B_i\bar f_{i,j}(\bar z)+f_{i,j}(B_j\bar z)\\
+D_{i,j}\left(R_{i}(\bar D_{i,j}\bar z) -R_{i}(\bar F_{i,j}\bar z)\right)& & +\left(f_{i,j}(\rho_j)-f_{i,j}(B_j\bar z)\right).
\end{eqnarray*}
Moreover, $F$ is linearized on $V({\cal I})$. Hence, both $\{D_{i,j}B_i\bar f_{i,j}(\bar z)+f_{i,j}(B_j\bar z)\}_Q$ and $\{R_{i}(\bar D_{i,j}\bar z) -R_{i}(\bar F_{i,j}\bar z)\}_Q$ vanish when $z^Q$ doesn't belong to ${\cal I}$. Hence, if $z^Q\not\in {\cal I}$, then we have
$$
(1-\mu_{i,j,k}\bar\mu^Q_{i,j})R_{Q,i,k}=\{\left(f_{i,j,k}(\rho_j)-f_{i,j,k}(B_j\bar z)\right\}_Q.
$$
But by induction, we have 
$$
\{\left(f_{i,j,k}(\rho_j)-f_{i,j,k}(B_j\bar z)\right\}_Q = \{Df_{i,j,k}(B_j\bar z)R_j+Df_{i,j,k}^2(B_j\bar z)R_j^2+\cdots\}_Q =0.
$$
Therefore, since $(1-\mu_{i,j,k}\bar\mu^Q_{i,j})\neq 0$, then we have $R_{Q,i,k}=0$. That is, 
$$
\rho_i(z) = B_i\bar z \mod \bar{\cal I}.
$$
\end{proof}
\begin{coros}
Under the assumptions of theorem \ref{theo-main}, there exists a complex analytic subvariety ${\cal S}$ passing through the origin and intersecting each totally real submanifold $M_i$. In good holomorphic coordinate system, ${\cal S}$ is a finite intersection of a finite union of complex hyperplane defined by complex coordinate subspaces :
$$
{\cal S}=\cap_i\cup_j\{z_{i_j}=0\}.
$$
The intersection $M_k\cap {\cal S}$ is then given by 
$$
M_k\cap {\cal S} = \left\{z\in \cap_i\cup_j\{z_{i_j}=0\}\,|\; B_k\bar z =z\right\}.
$$
\end{coros}
\begin{proof}
The complex analytic subvariety ${\cal S}$ is nothing but $V({\cal I})$. The trace of it on $M_i$ is the fixed points set of $\rho_i$ belonging to $V({\cal I})$. 
According to the previous theorem, the $\rho_i$'s are holomorphically and simultaneously linearizable on $V({\cal I})$. By assumptions, ${\cal I}$ is a monomial ideal so $V({\cal I})$ is a finite intersection of a finite union of hyperplane defined by coordinate subspaces :
$$
{\cal S}=\cap_i\cup_j\{z_{i_j}=0\}.
$$
 
\end{proof}
\begin{coros}
Assume that the family $M$ is non-resonant, $G$ is formally linearizable and $D$ is diophantine. Then, in a good holomorphic coordinates system, $M$ is composed of linear totally real subspaces
$$
\bigcup_i\left\{z\in\Bbb C^n\,|\; B_i\bar z = z\right\}.
$$
\end{coros}
\begin{rems}
If the family $M$ is non-resonant and if for all $(i,k)$, one of the eigenvalues $\mu_{i,j,k}$'s belong to the unit circle, then $G$ is formally linearizable. In fact, for any $Q\in \Bbb N_2^n$, any $1\leq i\leq m$, any $1\leq k\leq n$, there exists $1\leq j\leq m$ such that 
$$
\bar\mu_{i,j}^Q \neq \mu_{i,j,k}^{-1}=\bar \mu_{i,j,k}.
$$
This means precisely that $D$ is non-resonant in the classical sense. There is no obstruction to formal linearization.
\end{rems}
\begin{coros}
Let ${\cal I}$ be the ideal generated by the monomials $x^{R_1},\ldots,
x^{R_p}$ generating the ring $\widehat{\cal O}_n^D$ of formal invariants of
$D$. We assume that the non-linear centralizer of $D$ is generated by the same
monomials. If $D$ is diophantine on ${\cal I}$ then, in a good holomorphic coordinate system, we have 
$$
V({\cal I})=\{z\in (\Bbb C^n,0)\;|\; z^{R_1}=\cdots=z^{R_p}=0 \},
$$
and 
$$
\rho_{i|V({\cal I})}(z)= B_i\bar z.
$$
\end{coros}
\begin{coros}
Let us consider two totally real $n$-manifolds of $(\Bbb C^n,0)$ not
intersecting transversally at the origin. Assume that the $l$ first
eigenvalues of $DF(0)$ are one. Let $\mu^{R_1}=1,\ldots, \mu^{R_p}=1$ be the
other (i.e. $R_i\in \Bbb N^n$ and $|R_i|>1$) generators of resonant
relations. 
Let 
$$
V({\cal I})=\{z\in (\Bbb C^n,0)\;|\; z_1=\cdots =z_l=z^{R_1}=\cdots=z^{R_p}=0 \}.
$$
If $DF(0)$ is diophantine on $V({\cal I})$, then in good holomorphic
coordinate system,
$$
M_i\cap V({\cal I}) = \left \{z\in V({\cal I})\,|\; B_i\bar z = z\right\},\quad i=1,2.
$$
\end{coros}

\bibliographystyle{alpha}

\begin{thebibliography}{}

\end{thebibliography}


\begin{thebibliography}{Web03}

\bibitem[AG05]{gong-ahern-tg}
P.~Ahern and X.~Gong.
\newblock A complete classification for pairs of real analytic curves in the
  complex plane with tangential intersection.
\newblock {\em J. Dyn. Control Syst.}, 11(1):1--71, 2005.

\bibitem[CG97]{Ghys}
G.~Cairns and E.~Ghys.
\newblock The local linearization problem for smooth $sl(n)$-actions.
\newblock {\em Enseignement Math.}, 43, 1997.

\bibitem[Cha86]{chaperon-ast}
M.~Chaperon.
\newblock G\'eom\'etrie diff\'erentielle et singularit\'es de syst\`emes
  dynamiques.
\newblock {\em Ast\'erisque}, 138-139, 1986.

\bibitem[DG02]{delatte-gramchev}
David DeLatte and Todor Gramchev.
\newblock Biholomorphic maps with linear parts having {J}ordan blocks:
  linearization and resonance type phenomena.
\newblock {\em Math. Phys. Electron. J.}, 8:Paper 2, 27, 2002.

\bibitem[GS68]{stern-semi-simple}
V.V. Guillemin and S.~Sternberg.
\newblock Remarks on a paper of {H}ermann.
\newblock {\em Trans. Amer. Math. Soc.}, 130:110--116, 1968.

\bibitem[GS15]{stolo-gong1}
X. Gong and L. Stolovitch.
\newblock Real submanifolds of maximum complex tangent space at a {CR} singular point {I}.
\newblock to appear in {\it Invent. Math.}, 1--64, 2015.
	
\bibitem[GY99]{gram-yoshi}
T.~Gramchev and M.~Yoshino.
\newblock Rapidly convergent iteration method for simultaneous normal forms of
  commuting maps.
\newblock {\em Math. Z.}, 231:p. 745--770, 1999.

\bibitem[Kas13]{kasner1}
E.~Kasner.
\newblock Conformal geometry.
\newblock In {\em Proc. 5. Intern. Math. Congr.}, pages 81--87, 1913.

\bibitem[Kus67]{Kushnirenko}
A.G. Kushnirenko.
\newblock Linear-equivalent action of a semi-simple lie group in the
  neighbourhood of a stationary point.
\newblock {\em Funct. Anal. Appl.}, 1:89--90, 1967.

\bibitem[Mos90]{moser-circle}
J{\"u}rgen Moser.
\newblock On commuting circle mappings and simultaneous {D}iophantine
  approximations.
\newblock {\em Math. Z.}, 205(1):105--121, 1990.

\bibitem[Nak98]{nakai-toulouse}
I.~Nakai.
\newblock The classification of curvilinear angles in the complex plane and the
  groups of {$\pm$} holomorphic diffeomorphisms.
\newblock {\em Ann. Fac. Sci. Toulouse Math. (6)}, 7(2):313--334, 1998.

\bibitem[Pfe15]{Pfeiffer1}
G.~A. Pfeiffer.
\newblock On the conformal geometry of analytic arcs.
\newblock {\em American J.Math.}, 37:395--430, 1915.

\bibitem[P{\"o}s86]{Posch}
J.~P{\"o}schel.
\newblock {On invariant manifolds of complex analytic mappings near fixed
  points}.
\newblock {\em Expo. Math.}, 4,(1986),97-109, 1986.

\bibitem[R{\"u}s77]{russmann-ihes}
H.~R{\"u}ssmann.
\newblock On the convergence of power series transformations of analytic
  mappings near a fixed point into a normal form.
\newblock Preprint I.H.E.S. M/77/178, p.1-44, 1977.

\bibitem[R{\"u}s02]{russmann-diffeo}
H.~R{\"u}ssmann.
\newblock Stability of elliptic fixed points of analytic area-preserving
  mappings under the {B}runo condition.
\newblock {\em Ergodic Theory Dynam. Systems}, 22(5):1551--1573, 2002.

\bibitem[Sto94]{Stolo-dulac}
L.~Stolovitch.
\newblock {Sur un th\'{e}or\`{e}me de Dulac}.
\newblock {\em Ann. Inst. Fourier}, 44(5):1397--1433, 1994.

\bibitem[Sto00]{stolo-ihes}
L.~Stolovitch.
\newblock Singular complete integrabilty.
\newblock {\em Publ. Math. I.H.E.S.}, 91:p.133--210, 2000.

\bibitem[Sto05]{stolo-webster}
L.~Stolovitch.
\newblock Family of intersecting totally real manifolds of {$(\Bbb C^n,0)$} and
  {CR}-singularities, 2005.
\newblock http://front.math.ucdavis.edu/0506.5052.

\bibitem[Sto08]{Stolo-asi07}
L.~Stolovitch.
\newblock {Normal Forms of holomorphic dynamical systems}.
\newblock In W.~Craig, editor, {\em Hamiltonian dynamical systems and
  applications}, pages 249--284. Springer-Verlag, 2008.

\bibitem[Tr{\'e}03]{trepreau-tg}
J.-M. Tr{\'e}preau.
\newblock Discrimination analytique des diff\'eomorphismes r\'esonnants de
  {$(\Bbb C,0)$} et r\'eflexion de {S}chwarz.
\newblock {\em Ast\'erisque}, 284:271--319, 2003.
\newblock Autour de l'analyse microlocale.

\bibitem[Web03]{webster-inter}
S.~M. Webster.
\newblock Pair of intersecting real manifolds in complex space.
\newblock {\em Asian J. Math.}, 7(4):449--462, 2003.

\bibitem[Yoc95]{yoccoz-ast}
J.-C. Yoccoz.
\newblock Petits diviseurs en dimension 1.
\newblock {\em Ast\'{e}risque}, 231, 1995.

\end{thebibliography}
\def\cprime{$'$}

\end{document}